\documentclass[12pt,reqno]{amsart}
\usepackage{amsmath}
\usepackage{amsfonts}
\usepackage{amssymb,esint}
\usepackage{mathrsfs}
\usepackage{hyperref}
\usepackage{stmaryrd}
\usepackage{multirow}
\usepackage{verbatim}
\usepackage{cancel}
\usepackage{xcolor}

\usepackage[T1]{fontenc}
\usepackage[utf8]{inputenc}
\usepackage[english]{babel}
\usepackage[a4paper,top=2cm,bottom=2cm,left=2cm,right=2cm]{geometry}
\usepackage{graphicx}
\usepackage{amsthm}
\usepackage{braket}
\usepackage[autostyle]{csquotes}
\usepackage{pdfpages}
\hypersetup{
	colorlinks=true,
	citecolor=black,
	linkcolor=black,}
\usepackage[capitalise]{cleveref}
\usepackage{tikz-cd}
\usepackage{enumerate}
\usepackage{fancyhdr}
\usepackage{array}
\usepackage{mathabx}
\usepackage{xspace}

\theoremstyle{plain}
\newtheorem{theorem}{Theorem}[section]
\newtheorem{lemma}[theorem]{Lemma}
\newtheorem{corollary}[theorem]{Corollary}
\newtheorem{conjecture}[theorem]{Conjecture}
\newtheorem{proposition}[theorem]{Proposition}
\theoremstyle{definition}
\newtheorem{definition}[theorem]{Definition}
\newtheorem{remark}[theorem]{Remark}
\newtheorem{example}[theorem]{Example}

\newcommand{\R}{\mathbb R}
\newcommand{\N}{\mathbb N}

\newcommand{\G}{\mathbb G}

\newcommand{\Sf}{\mathbb S}

\newcommand{\galg}{\mathfrak g}

\newcommand{\spn}{\mathrm{span}}

\newcommand{\mh}{\mathrm{MinHeight}}
\newcommand{\vol}{\mathrm{vol}}

\newcommand{\abn}{\mathrm{Abn}}

\newcommand{\de}{\mathrm{d}}

\newcommand{\gr}{\mathrm{Gr}}
\newcommand{\g}{\mathfrak{g}}

\newcommand{\p}{\mathcal{P}}
\newcommand{\dcc}{d_{\mathrm{cc}}}
\newcommand{\deu}{d_{\mathrm{eu}}}
\newcommand{\ddt}{\frac{\de}{\de t}}
\newcommand{\sfeu}{\Sf_{\mathrm{eu}}}
\newcommand{\sfcc}{\Sf_{\mathrm{cc}}}

\newcommand{\twowedge}{\textstyle{\bigwedge^2} }

\author{Enrico Le Donne, Luca Nalon}
\address{D\'epartement de Math\'ematiques, Ch. du mus\'ee 23, 1700 Fribourg (CH)}
\email{luca.nalon@unifr.ch}
\email{enrico.ledonne@unifr.ch}
\thanks{Both authors are partially supported by the Swiss National Science Foundation
	(grant 200021-204501 `\emph{Regularity of sub-Riemannian geodesics and
		applications}')
	and by the European Research Council (ERC Starting Grant 713998 GeoMeG
	`\emph{Geometry of Metric Groups}').}

\begin{document}

\title{Euclidean rectifiability of sub-Finsler spheres in free-Carnot groups of step 2}

\maketitle

\begin{abstract}
    We consider $2$-step free-Carnot groups equipped  with sub-Finsler distances. We prove that the metric spheres are codimension-one rectifiable from the Euclidean viewpoint. The result is obtained by studying how the Lipschitz constant for the distance function behaves near abnormal geodesics.
\end{abstract}

\tableofcontents

\section{Introduction}
The regularity of spheres in sub-Riemannian and sub-Finsler structures is related to a peculiar phenomenon of these geometries: the existence of abnormal geodesics, for an introduction to sub-Riemannian geometry and its abnormal curves see \cite{Montgomery}. Indeed, in \cite{regularity}, the authors proved that, in the case of sub-Finsler Carnot groups, the absence of non-constant abnormal geodesic implies that spheres are boundaries of Euclidean Lipschitz domains. Conversely, sub-Finsler spheres with cusps do appear in Carnot groups with non-trivial abnormal geodesic. The presence of such cusps is related to the rate of convergence to the asymptotic cones of finitely generated nilpotent groups, see \cite{rate} for details. In this paper we address the regularity of spheres for the class of free-Carnot groups of step 2, in terms of Euclidean rectifiability.
\begin{definition}
[Euclidean  rectifiable subset]  
Let $\mathbb{M}$ be a smooth manifold of dimension $d$ and let $k$ be a natural number. 
    A subset $S\subseteq \mathbb{M}$  is {\em Euclidean $k$-rectifiable}, if 
$\mathcal{H}_{\mathrm{Riem}}^k(S)<\infty$ and there are countably many measurable subsets $E_n \subseteq \R^n$ and Lipschitz functions $f_n\colon E_n \subseteq \R^k \to \mathbb{M}$ such that
    \begin{equation}\label{eq_def_rect}
            \mathcal{H}_{\mathrm{Riem}}^k\left( S \,\, \setminus \, \bigcup^\infty_{n=0} f_n(E_n) \right)= 0  \, ,
        \end{equation}
    where $\R^k$ is endowed with the Euclidean metric, the space
    $\mathbb{M}$ is endowed with some   Riemannian metric,
    and the $k$-dimensional Hausdorff measure $\mathcal{H}^k_{\mathrm{Riem}}$ is taken with respect this Riemannian metric. 
\end{definition}
We remark that the notion of Euclidean rectifiability for bounded subsets of $\mathbb{M}$ is independent of the choice of the Riemannian metric. The main result of this article is the following.
\begin{theorem} \label{rectifiability}
    Let $\G$ be a sub-Finsler free-Carnot group of step $2$. Denote by $\dcc$ the corresponding Carnot-Carathéodory distance. Then the unit sphere, defined as
    \begin{equation*}
     \sfcc \coloneq \set{ g \in \G \, | \, \dcc(g,0_\G) = 1 } \, , 
    \end{equation*}
    is Euclidean $(d-1)$-rectifiable, where $d$ is the topological dimension of $\G$ and $0_\G$ is its identity element.
\end{theorem}
We refer to \cref{setting} for the notion of sub-Finsler free-Carnot group of step $2$ and Carnot-Carathéodory distance. We refer to \cref{proofs} for the proof of \cref{rectifiability}. In this context, abnormal geodesic are length-minimizing curves with a control that is singular with respect to the end-point map \eqref{def_endppointmap}, see \cref{section_4}. In Carnot groups of step $2$ it is known that the \emph{abnormal set} $\abn$, i.e., the set of points reached by abnormal curves passing through the identity element of the group, is contained in an algebraic variety of co-dimension $3$, see \cite{sard_step2}. Free-Carnot groups of step $2$ are a class for which this bound cannot be improved, in this sense they have the largest possible amount of abnormal curves within Carnot groups of step $2$. We stress that very few result concerning Euclidean rectifiability of spheres in sub-Finsler structures with non-trivial abnormal curves are present in the literature, all of which study the problem for a specific sub-Finsler norm, see \cite{regularity_cartan, sub-finsler_structures}. Nonetheless, the statement of \cref{rectifiability} has been known in the sub-Riemannian case. Indeed, sub-Riemannian structures of step $2$ have no Goh extremals, hence, by \cite[Proposition~4.2]{mass_transportation}, outside of the diagonal the Carnot-Carathéodory distance is Lipschitz with respect to the Riemannian distance. This implies that level sets of the distance from a fixed point, i.e.\ spheres, are rectifiable sets for almost every value. In Carnot groups, the existence of dilations ensures the result. This strategy, valid for sub-Riemannian $2$-step Carnot groups, does not generalize to sub-Finsler distances, since the distance may not be Euclidean Lipschitz, see \cite{rate}.

The strategy that we follow to prove \cref{rectifiability} starts by considering on every such a $\G$ a Riemannian metric $\deu$ coming from specific scalar products (we call them \emph{adequate}, see \eqref{adequate}) on $\G$, seen in exponential coordinates as \eqref{def_G}. Following \cite[Section~4.1]{regularity}, we interpret the unit sphere $\sfcc$ as the graph of the function $\dcc(0_\G, \cdot)$, restricted to the unit sphere $\sfeu$ of the distance $\deu$. We continue by building upon the fact that, in this setting, the abnormal set is an algebraic variety of codimension $3$ to get information on its Minkowski content. We use this information to study the asymptotic behaviour of the Hausdorff measure for tubular neighborhoods of $\abn \cap \sfeu$. The critical part of the argument is to find a local bound for the Lipschitz constant of $\dcc(0_\G, \cdot)$ at a given point on $\sfeu$, in terms of its distance from the abnormal set. In this regard, \cref{lippow} states that this Lipschitz constant is bounded (up to a uniform constant) by $\delta^{-1}$, where $\delta$ is the distance of that point from the abnormal set. Finally, we combine all the previous results and, by applying a suited version of the area formula for graph functions, we prove \cref{rectifiability}.

The scheme of reasoning that we summarized is rather general as all the results in \cref{section_3} arbitrarily hold for Carnot groups. Nonetheless, a strong version of the Sard property (as in \cite{sardproperty}) is needed in other to get a result analog to \cref{dimabn}. However, this kind of results are known only for Carnot groups of small rank and/or step, see \cite{sard_step2,sardproperty,codimension}, even though no example of Carnot groups where the abnormal set has co-dimension smaller than $3$ is known. A possibly harder phenomenon to study in wider classes of Carnot groups is the asymptotic local behaviour of the Lipschitz constant of $\dcc(0_\G, \cdot)$, with respect to the distance from the abnormal set. In the latter part of \cref{section_5} we exhibit an example of a sub-Finsler free-Carnot group of step $2$ for which the bound of \cref{lippow} is optimal. However, we are currently not able to provide examples of sub-Finsler Carnot groups of step $2$ where the local Lipschitz constant of $\dcc(0_\G, \cdot)$ diverge to infinity faster than the reciprocal of the distance from the abnormal set. We are then lead to state the following conjecture.

\begin{conjecture} \label{conjecture}
  Let $\G$ be a sub-Finsler Carnot group of step $2$, denote by $\dcc$ its Carnot-Carathéodory distance and by $\abn$ the abnormal set of $\G$. Fix on $\G$ a Riemannian metric $\deu$. Let $K$ be a compact subset that does not contain $0_\G$, e.g., $K=\sfeu$.
  Then, there exists a constant $C > 0$ such that, for every $\delta>0$, the function $\dcc(0_\G,\cdot)$ is $\left(C\delta^{-1}\right)$-Lipschitz (with respect to the metric $\deu$) when restricted to the set $\Set{ g \in 
  K
  \colon \deu(g,\abn) \ge \delta}$.
\end{conjecture}

If the latter statement were proved, by combining it with \cite[Theorem~1.1]{sard_step2} and the arguments of this article, the Euclidean rectifiability of sub-Finsler spheres will be addressed for all sub-Finsler Carnot groups of step $2$. In complete generality, we do not know if one should expect the following:

\begin{conjecture}
    Metric spheres in sub-Finsler Carnot groups are Euclidean codimension-one rectifiable.
\end{conjecture}

{\em Acknowledgments.} We are grateful to E. Hakavuori and C. Bodart for several interesting discussions and opinions on the approach to this problem.

\section{Sub-Finsler free-Carnot groups of step 2} \label{setting}

A \emph{Carnot group} is a simply connected Lie group with stratifiable Lie algebra. We recall that a real Lie algebra $\g$ is \emph{stratifiable of step $s$} if it admits a direct sum decomposition
\begin{equation*}
    \g = \g_1 \oplus \dots \oplus \g_s \, ,
\end{equation*}
into linear spaces such that $[\g_1,\g_i]=\g_{i+1}$ for every $i=1,\dots,s-1$, $[\g_1,\g_s]=\set{0}$, and $\g_s \neq 0$. Such a decomposition is called \emph{stratification}. Free-nilpotent Lie algebras of step $s$ on finite sets are examples of stratifiable Lie algebras. We call \emph{free-Carnot group of step $s$ and rank $n$} a Carnot group with free-nilpotent Lie algebra of step $s$ on a set with $n$ elements, they are unique up to isomorphism of Lie groups.

We consider the following model for a free-Carnot group of step $2$ and rank $n$: 
\begin{equation}\label{def_G}
\G\coloneq \left(\R^n \times \twowedge (\R^n), *\right) ,
\end{equation} where $\twowedge (\R^n)$ is the space of skew-symmetric contravariant tensors on $\R^n$ of rank $2$, i.e., the vector space generated by the elements of the form $v \wedge w \coloneq v \otimes w - w \otimes v$, with $v$ and $w$ ranging in $\R^n$. Hereafter we will refer to the elements of $\G$ both as sums $x+Y$ and pairs $(x,Y)$ where $x \in \R^n$ and $Y \in \twowedge (\R^n)$. The identification between $\R^n \oplus \twowedge (\R^n)$ and $\R^n \times \twowedge (\R^n)$ remains understood. The product law $*$ that we consider in \eqref{def_G} is
\begin{equation} \label{prod}
   (x_1, Y_1) * (x_2, Y_2) \coloneq \left(x_1+x_2,Y_1+Y_2+\tfrac{1}{2}(x_1\wedge x_2) \right).
\end{equation} 
Such a nilpotent Lie group has Lie algebra $\galg\coloneq(\R^n \times \twowedge (\R^n),[\cdot,\cdot])$ given by Lie bracket $[(v_1,\omega_1),(v_1,\omega_2)]=(0,v_1 \wedge v_2)$ for every $v_1,v_2 \in \R^n$ and $\omega_1,\omega_2 \in \twowedge(\R^n)$. This Lie algebra is free-nilpotent of step $2$ and generated by the standard basis of $\R^n$. A stratification $\g=\g_1 \oplus \g_2$ is given by $$\g_1\coloneq\R^n \qquad \text{ and } \qquad\g_2\coloneq\twowedge (\R^n).$$ We denote the identity element of the group $\G$ by $0_\G \coloneq (0,0)$. We stress that is this framework, the exponential map $\exp \colon \g \to \G$ is the identity map on $\R^n \times \twowedge (\R^n)$. We introduce the projection map 
\begin{equation*}
    \pi \colon \G \to \R^n , \qquad \pi(x,Y) \coloneq x,
\end{equation*}
which is a Lie group homomorphism $\pi \colon (\G,*) \to (\R^n,+)$.

For every $g=(x,Y) \in \G$, the left-translation map $L_g \colon \G_n \to \G_n$ is defined by $L_g(h)\coloneq g*h$. Its differential at the identity is the map $\de L_g \colon T_{0_\G}\G= \g \to T_g\G$ such that, for every $(v,\omega) \in \g$, with $v \in \R^n$ and $\omega \in \twowedge (\R^n)$, it holds
\begin{equation*}
    \de L_g (v,\omega) \stackrel{\eqref{prod}}{=} \frac{\de}{\de t} \left(x+tv, Y+t\omega+\tfrac{1}{2}(x \wedge tv)\right)\bigg|_{t=0} = \left(v,\omega+\tfrac{1}{2}(x\wedge v)\right) .
\end{equation*}
We consider $\g_1$ as a subspace of the tangent space $T_{0_\G}\G$ of $\G$ at $0_\G$ and on the distribution
\begin{equation*}
    \Delta \coloneq \bigcup_{g \in \G} \de L_g (\g_1) \subseteq T\G ,
\end{equation*}
we fix a left-invariant norm $\lVert \cdot \rVert_\mathrm{sf}$, i.e., $\lVert \de L_g v \rVert_\mathrm{sf}=\lVert v \rVert_\mathrm{sf}$ for every $g \in\G$ and $v \in \g_1$. We refer to $\Delta$ as the \emph{horizontal distribution} of $\G$. 
We refer to $\lVert \cdot \rVert_\mathrm{sf}$ as a \emph{Carnot sub-Finsler structure} on $\G$ and to such a couple $(\G,\lVert \cdot \rVert_\mathrm{sf})$, with $\G$ an in \eqref{def_G} with product \eqref{prod}, as
a {\em sub-Finsler free-Carnot group of step $2$ (and rank $n$)}. An absolutely continuous curve $\gamma \colon [0,1] \to \G$  is said to be {\em horizontal} if $\dot{\gamma}(t) \in \Delta$ for almost every $t \in [0,1]$.

\begin{remark}
In every sub-Finsler free-Carnot group
$(\G,\lVert \cdot \rVert_\mathrm{sf})$  of step $2$, 
there is a correspondence between horizontal curves from the origin  and elements of the space $L^1\coloneq L^1([0,1],\g_1)$ of integrable curves into the first stratum of $\G$. Namely, on the one hand, for every $u \in   L^1$ we  define the {\em horizontal curve $\gamma_u \colon [0,1] \to \G$ associated to the control $u$} as the unique solution to the differential equation
\begin{equation*}
\begin{cases}
    \gamma(0) = 0_\G \, ,\\
    \dot{\gamma_u}(t) = \de L_{\gamma(t)} \, u(t)\, , \quad \text{for a.e.\ $t \in [0,1]$}\, . 
\end{cases}
\end{equation*}
On the other hand, to each horizontal curve $\gamma \colon [0,1] \to \G$ such that $\gamma(0) = 0_\G$, we  define  the {\em control   $u \in L^1 $ associated to $\gamma$}   by
\begin{equation*}
    u(t) \coloneq \de L_{\gamma(t)^{-1}} \dot{\gamma}(t) \in T_{0_\G}\G \, , \quad \text{for a.e.\ $t \in [0,1]$.}
\end{equation*}
The product law \eqref{prod} allows us to write explicitly the equation for a curve of given control $u$:
\begin{equation} \label{curve}
  \gamma_u(t) = \int^t_0 u(s) \, \de s + \frac{1}{2} \int^t_0  \int^s_0 u(\tau) \,  \de \tau  \wedge u(s) \, \de s \, .  
\end{equation}
We define the length $\ell(\gamma)$ of a horizontal curve $\gamma \colon [0,1] \to \G$ of control $u$ as
\begin{equation*}
    \ell(\gamma) \coloneq \int^1_0 \lVert \dot{\gamma}(t) \rVert_\mathrm{sf} \,  \de t = \int^1_0 \lVert u(t) \rVert_\mathrm{sf} \, \de t \, .
\end{equation*}
We remark that $\ell(\gamma)$ is invariant under reparametrization of the interval $[0,1]$.
\end{remark}

We now define the so called \emph{Carnot-Carathéodory distance $\dcc$} on $\G$ induced by the sub-Finsler structure $\lVert \cdot \rVert_\mathrm{sf}$ as 
\begin{equation} \label{ccdistance}
    \dcc(g,h) \coloneq \inf\Set{ \ell(\gamma) \, : \, \text{$\gamma$ horizontal curve with $g,h \in \gamma([0,1])$}} ,
\end{equation}
which   induces the   manifold topology thanks to the Chow–Rashevskii theorem, \cite{Montgomery}. 

\begin{remark}
    Since $\lVert \cdot \rVert_\mathrm{sf}$ is left invariant, the distance $\dcc$ is left-invariant as well, i.e.,
    \begin{equation*}
        \dcc(h*g_1,h*g_2)=\dcc(g_1,g_2) \, , \quad \text{for every $g_1,g_2,h \in \G$}.
    \end{equation*}
\end{remark}

\begin{definition}
    Let $\G$ be a sub-Finsler free-Carnot group of step $2$ and let $g,h \in G$ be two points in $\G$. A \emph{geodesic between $g$ and $h$} is a horizontal curve $\gamma : [0,1] \to\G$ that is length-minimizing from $g$ to $h$, i.e., $\gamma(0)=g, \gamma(1)=h$ and 
    \begin{equation*}
        \dcc(g,h) = \ell(\gamma) \, .
    \end{equation*}
\end{definition}

We stress that the completeness of $\G$ ensure the existence of a geodesic between every two elements of $\G$, see \cite{Montgomery}. This means that the infimum in \eqref{ccdistance} is always attained.

In this paper we investigate the rectifiability of the sub-Finsler unit sphere at the origin:
\begin{equation*}
    \Sf_{\mathrm{cc}}\coloneq\set{ g \in \G \, | \, \dcc(0_\G,g)=1 } \, 
\end{equation*}
with respect to a Riemannian distance on $\G$. Since any two Riemannian metrics are bi-Lipschitz equivalent on compact sets, the result does not depend on the choice of the Riemannian distance. It turns out that, in our argument, it is convenient to consider on $\G$, seen as the vector space $\R^n \times \twowedge (\R^n)$, Euclidean distances given by specific scalar products, as we next introduce.
\begin{definition} \label{def_adequate}
    Let $\G=\R^n \times \twowedge (\R^n)$ be a free-Carnot group of step $2$. A scalar product $\langle \cdot , \cdot \rangle$ on $\G$ is said to be \emph{adequate} whenever $\R^n$ and $\twowedge (\R^n)$ are orthogonal with respect to $\langle \cdot , \cdot \rangle$ and 
    \begin{equation} \label{adequate}
       \langle v_1 \wedge w_1 , v_2 \wedge w_2 \rangle = \langle v_1, v_2 \rangle \cdot \langle w_1 , w_2 \rangle -\langle v_1 , w_2 \rangle \cdot \langle w_1 , v_2  \rangle \, ,
    \end{equation}
for every $v_1,v_2,w_1,w_2 \in \R^n$.    
\end{definition}
We observe that, on the one hand, every adequate scalar product on $\G=\R^n \times \twowedge (\R^n)$ is uniquely determined by its restriction to $\R^n$. On the other hand, we claim that every scalar product on $\R^n$ can be extended to an adequate scalar product on $\R^n \times \twowedge (\R^n)$. Indeed, formula \eqref{adequate} defines a bilinear form on $\twowedge (\R^n)$, which  is positively defined by Cauchy-Schwarz inequality:
    \begin{equation*} 
         \langle v\wedge w , v\wedge w \rangle = \langle v,v \rangle\cdot \langle w , w \rangle  - \langle v,w \rangle^2 \geq  0 ,
    \end{equation*}
    and the equality holds if and only if $v$ is parallel to $w$, that is $v\wedge w=0$.
    
Hereafter $\lVert \cdot \rVert_{\mathrm{eu}}$ will denote an Euclidean norm on $\G$ induced by an adequate scalar product.
The function $\deu$ represents the distance with respect to this Euclidean norm, and 
\begin{equation}\label{def_seu}
    \sfeu\coloneq\set{ g \in \G \, | \, \deu(0_\G,g)=1 }
\end{equation} is the corresponding unit sphere centered at $0_\G$.

On $\R^n$ the   norms $\lVert \cdot \rVert_{\mathrm{sf}}$ and $\lVert \cdot \rVert_{\mathrm{eu}}$
are equivalent, in the sense that there is a constant $C_1>0$ such that 
\begin{equation} \label{constantc1}
C_1^{-1} \lVert v \rVert_\mathrm{eu} \le \lVert v \rVert_\mathrm{sf} \le C_1  \lVert v \rVert_\mathrm{eu} \, ,\qquad \text{for every $v \in \R^n$.}   
\end{equation} Moreover,
since for $v \in \R^n$ the curve $t\mapsto tv$ is horizontal, we obtain
\begin{equation} \label{c1_estiamate}
    \dcc(0_\G,v) \le C_1\lVert v \rVert_\mathrm{eu} \, , \qquad \text{for every $v \in \R^n$.} \\
\end{equation}
Finally, since $\dcc$ is continuous with respect to the manifold topology, it has a minimum and a maximum on the compact set $\sfeu$. Therefore it exists $M>0$ such that \begin{equation} \label{costantM}
   M^{-1} \le\dcc(0_\G,g)\le M \, , \qquad \text{for every $g \in \sfeu$.}
\end{equation}

We conclude this section by discussing some properties of adequate scalar products.
\begin{remark} 
    For every adequate scalar product $\langle \cdot , \cdot \rangle$ on $\R^n \times \twowedge (\R^n)$, we claim that  
    \begin{equation} \label{wedge_norm_ineq}
        \lVert v \wedge w \rVert_\mathrm{eu} \le \lVert v \rVert_\mathrm{eu} \cdot \lVert w \rVert_\mathrm{eu} , \quad\text{ for every } v,w \in \R^n,
    \end{equation}
and the equality holds if and only if $v$ and $w$ are orthogonal. Indeed, by \eqref{adequate}, we compute    
    \begin{equation*} 
        \lVert v \wedge w \rVert_\mathrm{eu}^2 = \langle v\wedge w , v\wedge w \rangle = \langle v,v \rangle\cdot \langle w , w \rangle  - \langle v,w \rangle^2 \le \lVert v \rVert_\mathrm{eu}^2 \cdot \lVert w \rVert_\mathrm{eu}^2 \, ,
    \end{equation*}
and the latter estimate is an equality if and only if $\langle v,w \rangle=0$.
\end{remark}

\begin{remark} \label{wedgeorthogonality}
    Let $\langle \cdot , \cdot \rangle$ be an adequate scalar product on $\R^n \times \twowedge (\R^n)$, hereafter $\perp$ will denote its corresponding orthogonality relation. Moreover, if $\mathcal{P} \le \R^n$ and $\mathcal{W} \le \twowedge (\R^n)$ are linear subspaces, we use the notation:
    \begin{equation*}
        \mathcal{P}^{\perp_1}\coloneq \mathcal{P}^\perp \cap \R^n, \qquad 
        \mathcal{W}^{\perp_2}\coloneq \mathcal{W}^\perp \cap \twowedge (\R^n) .
    \end{equation*}
We claim that if $v,w \in \R^n$ and $\mathcal{P} \le \R^n$ is such that
at least one between $v$ and $w$ belongs to $\mathcal{P}^{\perp_1}$, then $v \wedge w \in \left( \twowedge (\mathcal{P})\right)^{\perp_2}$. Indeed, if $x_1,x_2 \in \mathcal{P}$ and $v \in \mathcal{P}^{\perp_1}$, then
\begin{equation*}
    \langle v \wedge w, x_1 \wedge x_2 \rangle = \langle v , x_1  \rangle \cdot \langle w , x_2  \rangle - \langle v , x_2  \rangle \cdot \langle w , x_1  \rangle = 0 \, .
\end{equation*}
The same argument holds if $w \in \mathcal{P}^{\perp_1}$.
\end{remark}

\section{The sphere as graph} \label{section_3}

We consider the group $\G$ as in \eqref{def_G}, with product \eqref{prod}, and $t \in \R$. We define \emph{the dilation of factor $t$} as the map
\begin{equation*}
    \delta_t \colon \G \to \G \, , \qquad (x,Y) \mapsto (tx,t^2Y) \, .
\end{equation*}
One can immediately check that $\delta_t$ is a Lie group isomorphism for every $t \neq 0$. Moreover, once a Carnot sub-Finsler structure $\lVert \cdot \rVert_\mathrm{sf}$ is fixed on $\G$, the map $\delta_t$ is a metric dilation of factor $t$ with respect to the Carnot-Carathéodory distance \eqref{ccdistance}, i.e.,
\begin{equation} \label{metric_dilation}
  \dcc(\delta_t(g),\delta_t(h))=t\dcc(g,h)\, , \qquad \text{for every $g,h\in\G$ and $t\in\R$.}  
\end{equation}
For every $g \in \G$, we define \emph{the dilation vector field} as 
\begin{equation} \label{dilation_flow}
    \widehat{\delta}(g) \coloneq \ddt  \delta_t(g) \bigg|_{t=1} \in T_g \G \, , \quad \text{for every $g \in \G$.}
\end{equation}

\begin{remark} \label{dilation_transverse}
Fix on $\G$ an adequate scalar product. We claim that the vector field $g \mapsto \widehat{\delta}(g)$ is 
transverse to the foliation of $\G \setminus \set{0}$ given by:
\begin{equation*}
    \Set{ \sfeu(r) }_{r \in (0,+\infty)} \, ,
\end{equation*}
where $\sfeu(r)$ is the sphere centered at $0_\G$ of radius $r$ with respect to the distance $\deu$ induced by the scalar product. Indeed, by computing the dilation vector field at $g=(x,Y)$ in coordinates, we get
\begin{equation*}
    \widehat{\delta}(x,Y) = \ddt  \delta_t(x,Y) \bigg|_{t=1} = \ddt (tx,t^2Y) \bigg|_{t=1} = (x,2Y) \, .
\end{equation*}
From \cref{def_adequate}, we recall that $\R^n$ and $\twowedge (\R^n)$ are orthogonal. We then compute the scalar product between $\widehat{\delta}(g)$ and the vector orthogonal to $T_g\sfeu$, namely $$\braket{(x,Y),(x,2Y)}_{\mathrm{eu}}= \lVert x \rVert^2_{\mathrm{eu}}+2\lVert Y \rVert^2_{\mathrm{eu}},$$ which is not null as long as $(x,Y) \neq 0_\G$. We conclude that $\widehat{\delta}(g)$ is transverse to $T_g\sfeu$ for every $g \in \G \setminus \set{0}$.
\end{remark}

\begin{remark} \label{diffeophi}
We claim that the map defined by
\begin{equation} \label{def_phi}
    \phi \colon \sfeu \times (0,+\infty) \to \G \setminus \set{0} \, , \qquad (p,t) \mapsto \delta_{t}^{-1}(p) \, .
\end{equation}
is a global diffeomorphism. Indeed, for every $g \in \G \setminus \set{0}$, the set
\begin{equation} \label{dilation_radius}
    r_g \coloneq \set{ \delta_t(g) \, \colon \, t \in (0,+\infty)}
\end{equation}
is connected, it has $0_\G$ as an accumulation point, and it escapes every compact set of $\G$. This implies that $r_g \cap \sfeu$ is non-empty and therefore $g$ is in the image of $\phi$, this proves $\phi$ is surjective. Moreover,
the map $t \mapsto \lVert \delta_t(g)\rVert_\mathrm{eu}$ is smooth from $(0,+\infty)$ to $\R$ and by \cref{dilation_transverse} its derivative is never zero, therefore it is injective. This implies that $r_g \cap \sfeu$ has at most one element, hence $\phi$ is injective. Finally, $\phi$ is a local diffeomorphism by \cref{dilation_transverse}.
\end{remark}


\begin{remark} \label{sphere_graph}
    Let $S \coloneq \set{(p,d_0(p)) \, | \, p \in \sfeu}$ be the graph of the function $d_0(\cdot) \coloneq \dcc(0_\G, \cdot)$ restricted to $\sfeu$. We claim that, for $\phi$ as in \eqref{diffeophi}, one has
\begin{equation*}
    \sfcc=\phi(S) \, .
\end{equation*}
Indeed, by \cref{diffeophi}, each set $r_g$, defined in \eqref{dilation_radius}, has precisely one point in $\sfeu$ and, since $\delta_t$ is a metric dilation of factor $t$ that fixes $0_\G$, also one point in $\sfcc$. Moreover, for every $p \in \sfeu$, we have 
\begin{equation*}
    \dcc(0_\G, \phi(p,d_0(p)) \stackrel{\eqref{metric_dilation}}{=} d_0(p)^{-1}\cdot \dcc(0_\G,p) =1 \, .
\end{equation*}
This implies that the map $p \mapsto \phi(p,\dcc(p))$ is a bijection from $\sfeu$ to $\sfcc$.
\end{remark}


Hereafter, if $f$ is a Lipschitz function between metric spaces, the number $\mathrm{Lip}(f)$ will denote the Lipschitz constant of $f$. As already mentioned, the goal of this paper is to study the rectifiability of $\sfcc$. Since we have expressed it as the graph of a function, this motivates the need for the following result.

\begin{proposition} \label{areaformula}
    Let the group $\G$ be as in \eqref{def_G}, the set $\sfeu$ as in \eqref{def_seu} and the map $\phi$ as in \eqref{def_phi}. Denote by $d$ the topological dimension of $\G$ and by $d_0(\cdot)$ the function $\dcc(0_\G, \cdot)$. There exists a constant $C_2$ such that for every Borel set $E \subseteq \sfeu$, it holds
    \begin{equation*}
        \mathcal{H}^{d-1}_{\mathrm{eu}}(\phi(\mathrm{gr} \, d_0|_E)) \le C_2 \left( \mathrm{Lip}(d_0|_E) + 1\right)\mathcal{H}^{d-1}_{\mathrm{eu}}(E) \, ,
    \end{equation*}
    where $\mathrm{gr} \, d_0|_E \coloneq \set{(x,d_0(x)) \, | \, x \in E} \subseteq \sfeu \times (0,+\infty)$, the measure $\mathcal{H}_\mathrm{eu}$ is the Hausdorff measure with respect to $\deu$, and on $\sfeu$ we consider the distance function given by the restriction of $\deu$.
\end{proposition}

To prove \cref{areaformula}, we recall the following corollary of the area formula \cite[Theorem~3.2.3]{Federer} applied to graph functions.

\begin{corollary} \label{area_formula_real}
Let $A \subseteq \R^d$ be a Borel set and $\psi \colon A \subseteq \R^d \to \R$ be a Lipschitz function. Then 
    \begin{equation*}
        \mathcal{H}^d(\mathrm{gr} \, \psi) \le \left( \mathrm{Lip}(\psi) + 1\right)\mathcal{H}^d(A) \, ,
    \end{equation*}
    where $\mathrm{gr} \, \psi \coloneq \set{(x,\psi(x)) \, | \, x \in A} \subseteq \R^{d+1}$ and $\mathcal{H}^d$ is the Hausdorff measure on $\R^d$ and $\R^{d+1}$. 
\end{corollary}


\begin{proof}[Proof of \cref{areaformula}] This is a simple exercise, which we spell out for completeness. Fix a Borel set $E \subseteq \sfeu$. We recall from \eqref{costantM} that, for some $M>0$, we have that $M^{-1} \le\dcc(0_\G,g)\le M$ for every $g \in \sfeu$. In particular $\mathrm{gr} \, d_0|_E \subseteq \sfeu \times [M^{-1},M]$. By \cref{diffeophi}, the map $\phi \colon \sfeu \times (0,+\infty) \to \G\setminus\set{0}$ is a diffeomorphism, therefore $\phi$ is Lipschitz on compact sets. We define $L_1$ to be the Lipschitz constant of $\phi$ when restricted to $\sfeu \times [M^{-1},M]$. We recall that $(\sfeu,\deu)$ is locally bi-Lipschitz equivalent to $\R^{d-1}$. By the compactness of $\sfeu$, we can then consider a finite collection of closed subsets $B_1,\dots,B_N \subseteq \R^{d-1}$ and bi-Lipschitz maps $h_i \colon B_i \to h_i(B_i) \subseteq \sfeu$ such that $$\bigcup_{i=1}^N \, h_i(B_i) = \sfeu\, .$$ 
Let $L_2 \ge 1$ be a number greater than $\mathrm{Lip}(h_i)$ and $\mathrm{Lip}(h_i^{-1})$ for every $1 \le i \le N$.
For every $1 \le i \le N$, we consider the functions defined as $\psi_i \coloneq d_0 \circ h_i$ from $B_i$ to $[M^{-1},M]$. We consider on $\sfeu \times (0,+\infty)$ the product metric $\sqrt{\deu^2 + \lvert \cdot \rvert^2}$ and the Hausdorff measure $\mathcal{H}^d$ with respect to that metric. We observe that the maps defined as
\begin{equation*}
    \Psi_i \colon  \mathrm{gr} \, \psi_i \to \mathrm{gr} \, d_0|_{h_i(B_i)} \, , \qquad
    (x,\psi_i(x)) \mapsto (h_i(x),\psi_i(x)) \, ,
\end{equation*}
are $L_2$-Lipschitz.

Let $A_1,\dots,A_N \subseteq \R^{d-1}$ be Borel sets such that $A_i \subseteq B_i$ for every $1 \le i \le N$ and
\begin{equation*}
    E = \bigsqcup_{i=1}^N h_i(A_i) \, .
\end{equation*}
Finally, we can combine the previous discussion with \cref{area_formula_real} to compute
\begin{align*}
\mathcal{H}^{d-1}_{\mathrm{eu}}(\phi&(\mathrm{gr} \, d_0|_E)) \le L_1^{d-1} \mathcal{H}^{d-1}(\mathrm{gr} \, d_0|_E)=L_1^{d-1} \mathcal{H}^{d-1}\left(\bigcup_{i=1}^N \, \mathrm{gr} \, d_0|_{h_i(A_i)}\right) \\&=L_1^{d-1} 
\sum_{i=1}^N \mathcal{H}^{d-1}\left(\mathrm{gr} \, d_0|_{h_i(A_i)}\right) \le L_1^{d-1} L_2^{d-1} 
\sum_{i=1}^N \mathcal{H}^{d-1}\left(\mathrm{gr} \, \psi_i|_{A_i}\right) 
\\ &\le L_1^{d-1} L_2^{d-1} 
\sum_{i=1}^N \left( \mathrm{Lip}(\psi_i) + 1\right)\mathcal{H}^{d-1}\left(A_i\right) \le L_1^{d-1} L_2^{d} \left( \mathrm{Lip}(d_0|E) + 1\right)  \sum_{i=1}^N \mathcal{H}^{d-1}\left(A_i\right)
\\&\le L_1^{d-1} L_2^{2d-2} \left( \mathrm{Lip}(d_0|_E) + 1\right) \sum_{i=1}^N \mathcal{H}^{d-1}_{\mathrm{eu}}\left(h_i(A_i)\right) \le L_1^{d-1} L_2^{2d-2} \left( \mathrm{Lip}(\psi) + 1\right) \mathcal{H}^{d-1}_{\mathrm{eu}}\left(E\right) \, .
\end{align*}
The statement is then proved by setting $C_2 \coloneq L_1^{d-1} L_2^{2d-1}$.
\end{proof}

\section{Abnormal set and Minkowski content} \label{section_4}

The \emph{abnormal set} $\abn(\G)$ of the Carnot group $\G$ as in \eqref{def_G} is, by definition, the set of critical values of the end-point map $\mathrm{End}$ defined as
\begin{equation} \label{def_endppointmap}
  \mathrm{End} \colon L^1([0,1],\g_1) \to \G, \qquad u \mapsto \gamma_u(1),  
\end{equation}
where $\gamma_u$ is defined as in \eqref{curve}. In this case, the abnormal set can be described using Grassmannians. We recall that the \emph{Grassmannian} $\mathrm{Gr}(n,k)$, with $n,k \in \N$, is the family of all $k$-dimensional vector subspaces of $\R^n$. From \cite[Equation~3.9]{sardproperty}, the abnormal set $\abn(\G)$ can be written as
\begin{equation} \label{abnormalset}
    \abn(\G) = \bigcup \Set{ \mathcal{P} \times \twowedge (\mathcal{P}) \, \colon \mathcal{P} \in \mathrm{Gr}(n,n-2)} .
\end{equation}

In the following result we summarize some properties of the abnormal set \eqref{abnormalset}, referring to \cite{sardproperty} for the details. 
\begin{proposition} \label{codim3} \cite[Theorem 3.14, Theorem 3.15]{sardproperty}
    Let $\G$ be as in \eqref{def_G}, then the abnormal set \eqref{abnormalset} is an algebraic variety of co-dimension $3$ within $\G$ and it is invariant by dilations, i.e., 
    \begin{equation*}
        \delta_t(\abn(\G)) = \abn(\G) \, , \qquad \text{for every $t \in (0,+\infty)$.}
    \end{equation*}
\end{proposition}

We recall from \cite[Section~3.4]{Benedetti} that a bounded algebraic variety $\mathcal{A}\subseteq \R^d$ of dimension $m$ is the finite union of smooth compact sub-varieties of $\R^d$ of dimension not greater than $m$. In particular $\mathcal{A}$ is closed and $m$-rectifiable in the stronger sense of \cite[Definition~3.2.14]{Federer}. Therefore, by \cite[Theorem~3.2.39]{Federer}, its $m$-dimensional Minkowski content equals its $m$-dimensional Hausdorff measure, which is finite. The following proposition states that an $m$-rectifiable set has finite $m$-dimensional upper Minkowski content and it is a simple consequence of \cite[Theorem~3.2.39]{Federer}.

\begin{proposition} \label{Minkowski_content}
    Let $0 \le k \le d$ be integer numbers and $\mathcal{S} \subseteq \R^d$ be a $(d-k)$-rectifiable set in the stronger sense of \cite[Definition~3.2.14]{Federer}, then, 
    for every $\varepsilon_0>0$, there exists $c > 0$ such that 
    \begin{equation*}
        \mathcal{H}^{d}\left(\set{x \in \R^d \, | \, \deu(x,\mathcal{S})\le\varepsilon}\right) \le c \cdot \varepsilon^k \, , \qquad \text{for every $\varepsilon\in (0,\varepsilon_0)$.}
    \end{equation*}
\end{proposition}

A similar result allows us to estimate the $(d-1)$-dimensional Hausdorff measure of the following set
\begin{equation} \label{def_A_delta}
A_\delta \coloneq \set{p \in \sfeu \, | \, \deu(p,\abn(\G))\le \delta}.\end{equation}

\begin{proposition} \label{dimabn}
    Let $\G$ be as in \eqref{def_G} and $d$ be its topological dimension. Fix on $\G$ an adequate scalar product and denote by $\deu$ the induced distance. With respect to $\deu$, consider the sphere \eqref{def_seu} and its subsets \eqref{def_A_delta}.
    Then there exists a constant $C_3$ such that 
    \begin{equation*}
        \mathcal{H}^{d-1}(A_\delta) \le C_3\delta^3 \,, \qquad \text{for every $\delta\in (0,1)$.}
    \end{equation*}
\end{proposition}

\begin{proof}
We firstly remark that is not restrictive to prove the statement for $\delta$
smaller than a given $\delta_0>0$. We recall from \cref{diffeophi} that for every $g \in \G \setminus \set{0}$ the set $r_g \cap \sfeu$, where $r_g$ is defined as in \eqref{dilation_radius}, is a singleton. Denote by $\hat{g}$ the element of that singleton. We then define the map
\begin{equation} \label{map_pi}
     \sigma \colon \Set{g \in \G \, | \, \deu(g,\sfeu) \le 1/2} \to \sfeu \, , \qquad
     g \mapsto \hat{g} \, .
\end{equation}

Consider the diffeomorphism $\phi$ as in \eqref{def_phi}. We observe that the map $\sigma$ is conjugated, thought $\phi$, to the map
\begin{equation*}
    f \colon \sfeu \times (0,+\infty) \to \sfeu \times (0,+\infty) \, , \qquad (x,t) \mapsto (x,1) \, ,
\end{equation*}
which is $1$-Lipschitz if we consider on $\sfeu \times (0,+\infty)$ the product metric $\sqrt{\deu^2 + \lvert \cdot \rvert^2}$. Since $\phi$ is bi-Lipschitz when restricted to compact sets, we conclude that $\sigma$ is Lipschitz as well. We denote the Lipschitz constant of $\sigma$ by $L_1$. 

From \cref{codim3}, we get that $\abn(\G)$ is closed and invariant by dilation. Then, for every $\delta \le 1/2$ and $p \in A_\delta$, we have $\delta \ge \deu(p,\abn(\G))=\deu(p,q)$ for some $q \in \abn(\G)$. The estimate $\deu(p,\sigma(q))=\deu(\sigma(p),\sigma(q))\le L_1\deu(p,q)\le L_1\delta$ implies the following inclusion:
\begin{equation} \label{a_delta_inclusion}
    A_\delta \subseteq \set{p \in \sfeu \, | \, \deu(p,\abn(\G)\cap\sfeu)\le L_1\delta}, \qquad \text{for every $\delta \le 1/2$.}
\end{equation}

Since $\sfeu$ is compact and locally bi-Lipschitz equivalent to $\R^{d-1}$, we can consider a finite collection of open subsets $B_1,\dots,B_N \subseteq \R^{d-1}$ and bi-Lipschitz maps $h_i \colon B_i \to h_i(B_i) \subseteq \sfeu$ such that $\Set{h_i(B_i) \, | \, 1 \le i \le N}$ is a covering of $\sfeu$. Let $\delta_0>0$ be such that, for every $p,q \in \sfeu$, 
\begin{equation} \label{delta_0}
\deu(p,q)\le L_1\delta_0 \quad \Longrightarrow \quad \text{$p,q \in h_i(B_i)$ for some $1 \le i \le N$.}
\end{equation}
Let $L_2 \ge 1$ be a constant greater than $\mathrm{Lip}(h_i)$ and $\mathrm{Lip}(h_i^{-1})$ for every $1 \le i \le N$. We define the set
\begin{equation*} 
E_{i,\delta}\coloneq \Set{\Bar{p} \in B_i \, | \, d(\Bar{p},h_i^{-1}(\abn(\G) \cap h_i(B_i)) \le L_1 L_2 \delta},
\end{equation*}
and we claim that, for every $\delta \le \delta_0$,
\begin{equation} \label{hi_inclusion}
\set{p \in \sfeu \, | \, \deu(p,\abn(\G)\cap\sfeu)\le L_1\delta} \subseteq \bigcup_{i=1}^N h_i\left(E_{i,\delta}\right) \, .    
\end{equation}
Indeed, fix $\delta \le \delta_0$ and $p \in \sfeu$ such that $\deu(p,\abn(\G)\cap\sfeu)\le L_1\delta$. Then $\deu(p,q)\le L_1\delta_0$ for some $q \in \abn(\G)\cap\sfeu$. By \eqref{delta_0}, there exists $1 \le i \le N$ such that $p,q \in h_i(B_i)$. Define $\Bar{p}\coloneq h_i^{-1}(p)$ and $\Bar{q}\coloneq h_i^{-1}(q)$. Then $\Bar{q} \in h_i^{-1}(\abn(\G) \cap \sfeu)$ and
\begin{equation*}
    d(\Bar{p},\Bar{q}) \le L_2 \deu(h_i(\Bar{p}),h_i(\Bar{q})) \le L_2\deu(p,q) \le L_1L_2\delta \, .
\end{equation*}
Therefore $p \in h_i\left(E_{i,\delta}\right)$ and the inclusion \eqref{hi_inclusion} is proved.

From \cref{codim3}, the set $\abn(\G)$ is an algebraic variety of dimension $d-3$ and invariant by dilations. Moreover, from \cref{dilation_transverse}, the manifold $\sfeu$ is transverse to the dilation flow \eqref{dilation_flow}. As a consequence of the implicit function theorem, we infer that $\abn(\G) \cap \sfeu$ is a bounded algebraic variety of dimension $d-4$, hence it is $(d-4)$-rectifiable. We observe that each $h_i^{-1}(\abn(\G) \cap h_i(B_i))$ is the image by a Lipschitz map of the $(d-4)$-rectifiable set $\abn(\G) \cap h_i(B_i)$. By \cref{Minkowski_content} with $\varepsilon_0 \coloneq L_1L_2\delta_0$, there exists $c>0$ such that, for every $1 \le i \le N$,
\begin{equation} \label{estimate_eidelta}
    \mathcal{H}^{d-1}(E_{i,\delta}) \le c L_1^{3}L_2^{3} \delta^3 \, , \qquad \text{for every $\delta \le \delta_0$.}
\end{equation}
Let $\delta\le\delta_0$, we can then compute
\begin{align*}
   \mathcal{H}^{d-1}\left(A_{\delta}\right) &\stackrel{\eqref{a_delta_inclusion}}{\le} \mathcal{H}^{d-1}\left(\set{x \in \sfeu \, | \, \deu(x,\abn(\G)\cap\sfeu)\le L_1\delta}\right) \stackrel{\eqref{hi_inclusion}}{\le}\sum_{i=1}^N \mathcal{H}^{d-1}\left(h_i(E_{i,\delta})\right) \\
   &\le \sum_{i=1}^N L_2^{d-1}\mathcal{H}^{d-1}(E_{i,\delta}) \stackrel{\eqref{estimate_eidelta}}{\le} \sum_{i=1}^N cL_2^{d-1} L_1^{3}L_2^{3} \delta^3 \le cNL_1^3L_2^{d+2}\delta^3 \, .
\end{align*}
The statement is proved by setting $C_3 \coloneq cNL_1^3L_2^{d+2}$.
\end{proof}

As another consequence of \eqref{abnormalset}, we have that
\begin{equation} \label{distinf}
    \deu\left(g, \abn(\G)\right)= \inf \Set{ \deu\left(g, \mathcal{P} \times \twowedge (\mathcal{P})\right) \, \colon \mathcal{P} \in \mathrm{Gr}(n,n-2)} .
\end{equation}
In the latter part of this section we investigate some properties for a geodesic $\gamma$ starting from the origin assuming that the final point $\gamma(1)$ is sufficiently far away from $\abn(\G)$.

\begin{remark} \label{distancebound}
Let $\gamma \colon [0,1] \to \G$ be a horizontal curve with control $u \in L^1\left([0,1],\g_1\right)$. Given $\mathcal{P} \in \gr(n,k)$ we can decompose $u(t)=u_1(t)+u_2(t)$ with $u_1(t) \in \mathcal{P}$ and $u_2(t) \in \mathcal{P}^{\perp_1}$. We claim that for every $t \in [0,1]$ we have
\begin{equation}\label{distancegammap}
    \deu\left(\pi(\gamma(t)),\p\right) = \bigg\lVert \int^t_0 u_2(s) \, \de s \bigg\rVert_\mathrm{eu} \, , \qquad \text{for every $t \in [0,1]$.}
\end{equation}
Indeed, 
\begin{align*}
    \deu\left(\pi(\gamma(t)),\p\right) &\stackrel{\eqref{curve}}{=} \deu \left( \int^t_0 u(s) \, \de s,\p \right) = \deu \left( \int^t_0 u_1(s) \, \de s +\int^t_0 u_2(s) \, \de s,\p \right) \\ &= \deu \left( \int^t_0 u_2(s) \, \de s, 0_\G \right) = \bigg\lVert \int^t_0 u_2(s) \, \de s \bigg\rVert_\mathrm{eu} \, .
\end{align*}
\end{remark}

For the rest of this section, we fix a Carnot sub-Finsler structure $\lVert \cdot \rVert_\mathrm{sf}$ on the group \eqref{def_G}.

\begin{lemma} \label{distancewedge}
There exists a constant $K$ such that for every $\varepsilon >0$, every sub-Finsler geodesic $\gamma \colon [0,1] \to \G$, with $\gamma(0)=0_\G$ and $\gamma(1) \in \sfeu$, and every $\p \le \R^n$, if $\deu(\pi(\gamma(t)),\p) < \varepsilon$ for every $t \in [0,1]$, then $\deu\left(\gamma(1), \p \times \textstyle{\twowedge} (\p)\right) < K\varepsilon$.
\end{lemma}

\begin{proof}
Let $C_1$ and $M$ be the constants given in \eqref{constantc1} and \eqref{costantM}, respectively. We prove the result for $K \coloneq 1 + 2C_1M$. Let $\varepsilon$, $\gamma$, and  $\p$  be as in the assumptions of the lemma. We can parameterize $\gamma$ with constant speed, with control $u \colon [0,1] \to \R^n$. Since $\gamma$ is a geodesic and $\gamma(1) \in \sfeu$, we have that 
\begin{equation*}
    \lVert u(t) \rVert_\mathrm{sf} \equiv \ell(\gamma) = \dcc(0_\G,\gamma(1)) \stackrel{\eqref{costantM}}{\le} M \, , \qquad \text{for every $t \in [0,1]$.}
\end{equation*}
The latter implies that 
\begin{equation} \label{c1M}
   \lVert u(t) \rVert_\mathrm{eu} \le C_1M \, , \qquad \text{for every $t \in [0,1]$.}
\end{equation} 

Let us now write $u(t) = u_1(t) + u_2(t)$, where $u_1(t) \in \p$ and $u_2(t) \in \p^{\perp_1}$ for every $t \in [0,1]$. By \cref{distancebound}, the condition on $\gamma$ implies that, for every $0 \le a \le b \le 1$,
\begin{equation} \label{integral_norm_estimate}
    \bigg\lVert \int^b_a u_2(s) \, \de s \bigg\rVert_\mathrm{eu} \le \bigg\lVert \int^b_0 u_2(s) \, \de s \bigg\rVert_\mathrm{eu} + \bigg\lVert \int^a_0 u_2(s) \, \de s \bigg\rVert_\mathrm{eu}
    \stackrel{\eqref{distancegammap}}{\le} 2 \cdot \sup_{t \in [0,1]} \deu(\pi(\gamma(t)),\p) \le 2\varepsilon \, .
\end{equation}

We use the decomposition $\g_1 = \mathcal{P} \oplus \mathcal{P}^{\perp_1}$ to write  
\begin{align*}
    \gamma_u(1) &\stackrel{\eqref{curve}}{=} \int^1_0 u_1(t) \de t + \frac{1}{2}\int^1_0  \int^t_0 u_1(s) \,  \de s \wedge u_1 (t) \, \de t + \\ &+\int^1_0 u_2(t) \de t + \frac{1}{2} \int^1_0  \int^t_0 u_1(s) \,  \de s  \wedge u_2(t) \, \de t + \frac{1}{2}\int^1_0  \int^t_0 u_2(s) \,  \de s \wedge u(t) \, \de t \, .
\end{align*}
We observe that the first line belongs to $\p \times \twowedge(\p)$ and, by \cref{wedgeorthogonality}, the second line belongs to $\left(\p \times \twowedge(\p) \right)^\perp$.
We can then write
\begin{align*}
    \deu&\left(\gamma(1), \p \times \textstyle{\twowedge} (\p) \right)= \\
    &= \left\lVert \int^1_0 u_2(t) \de t + \frac{1}{2} \int^1_0  \int^t_0 u_1(s) \,  \de s  \wedge u_2(t) \, \de t + \frac{1}{2}\int^1_0  \int^t_0 u_2(s) \,  \de s \wedge u(t) \, \de t  \right\rVert_\mathrm{eu} \\
    &  \le \left\lVert \int^1_0 u_2(t) \de t \right\rVert_\mathrm{eu} + \frac{1}{2} \left\lVert \int^1_0  \int^t_0 u_1(s) \,  \de s \wedge u_2(t) \, \de t \right\rVert_\mathrm{eu} + \frac{1}{2} \left\lVert \int^1_0  \int^t_0 u_2(s) \,  \de s \wedge u(t) \, \de t \right\rVert_\mathrm{eu} .
\end{align*}
The first term is strictly less than $\varepsilon$ thanks to \eqref{integral_norm_estimate}. We can apply Fubini to the second term and get
\begin{align*}
    \frac{1}{2} \left\lVert \int^1_0  \int^t_0 u_1(s) \,  \de s  \wedge u_2(t) \, \de t \right\rVert_\mathrm{eu} &= \frac{1}{2} \left\lVert \int^1_0  u_1(s)  \wedge  \int^1_s u_2(t) \, \de t \, \de s \right\rVert_\mathrm{eu} \\ &\stackrel{\eqref{wedge_norm_ineq}}{\le} \frac{1}{2} \int^1_0 \left\lVert u_1(s) \right\rVert_\mathrm{eu} \cdot \left\lVert \int^1_s u_2(t) \, \de t \right\rVert_\mathrm{eu} \de s \\ 
    &\stackrel{\eqref{integral_norm_estimate}}\le \frac{1}{2} \int^1_0 \left\lVert u_1(s) \right\rVert_\mathrm{eu} \cdot 2\varepsilon \, \de s\,  \stackrel{\eqref{c1M}}{\le} C_1M\varepsilon \, .
\end{align*}
A similar estimates allows us to bound the third term, namely,
\begin{equation*}
\frac{1}{2} \left\lVert \int^1_0  \int^t_0 u_2(s) \,  \de s  \wedge u(t) \, \de t \right\rVert_\mathrm{eu} \stackrel{\eqref{wedge_norm_ineq}}{\le} \frac{1}{2} \int^1_0 \left\lVert \int^t_0 u_2(s) \, \de s \right\rVert_\mathrm{eu} \cdot \left\lVert u_1(t) \right\rVert_\mathrm{eu} \de t \stackrel{\eqref{c1M},\eqref{integral_norm_estimate}}{\le} C_1M\varepsilon \, .
\end{equation*}
By combining everything together, we finally get
\begin{equation*}
    \deu\left(\gamma(1), \p \times \textstyle{\twowedge} (\p)\right) < 
    \varepsilon + C_1M \varepsilon + C_1M \varepsilon \le \varepsilon \left( 1 + 2C_1M \right) \, .
\end{equation*}
The proposition is then proved by setting $K \coloneq 1 + 2C_1M$.
\end{proof}

\begin{corollary} \label{abndist}
For every sub-Finsler free-Carnot group of step $2$, as in \eqref{def_G}, there exists a constant $C_4>0$ with the following property:

For every $\delta>0$ and every geodesic $\gamma \colon [0,1] \to \G$, with $\gamma(0)=0_\G$ and $\gamma(1) \in \sfeu$, such that $\deu(\gamma(1),\abn(\G))\ge\delta$, then $\Set{\pi(\gamma(t)) \, | \, t \in [0,1]}$ is not contained in the $C_4\delta$-neighbourhood of any $\p \in \gr(n,n-2)$, i.e., for every $\p \in \gr(n,n-2)$ there exists $t \in [0,1]$ such that $\deu(\pi(\gamma(t)),\p) \ge C_4\delta$.
\end{corollary}

\begin{proof}
    Let $K$ be the constant from \cref{distancewedge}, we prove that the statement holds for $C_4 \coloneq 1/K$. Suppose, by contradiction, that there exists some $\p \in \gr(\R^n,n-2)$ for which $\deu(\pi(\gamma(t)),\p) < C_4\delta$ for every $t \in [0,1]$. Recall, from \eqref{abnormalset}, that $\p \times \twowedge (\p) \subseteq \abn(\G)$. Then, by \cref{distancewedge} with $\varepsilon\coloneq C_4\delta$, we get that $$\deu(\gamma(1),\abn(\G)) \stackrel{\eqref{distinf}}{\le} \deu\left(\gamma(1), \p \times \textstyle{\twowedge} (\p)\right) < KC_4\delta = \delta \, ,$$
    hence a contradiction. 
\end{proof}

\section{Minimal height and rectifiability of spheres} \label{section_5}

The consequence of \cref{abndist} can be informally rephrased by saying that the projection of each geodesic is not close to any vector subspace of co-dimension $2$. This property is closely related to the notion of minimal height of a configuration of points, which we next recall from \cite{blow}.

\subsection{Some remarks on minimal height}
Hereafter, if $(a_1,\dots,a_m)$ is a $m$-tuple of points in $\R^n$, we will denote by $\spn(a_1,\dots,a_m)$ the subspace generated by those vectors. We will denote by $(a_1,\dots,\widehat{a}_j,\dots,a_m)$ the $(m-1)$-tuple obtained from $(a_1,\dots,a_m)$ by removing the $j$-th element. 

\begin{definition}
The \emph{$m$-dimensional volume} of an $m$-tuple $(a_1,\dots,a_m)$ of vectors in the Euclidean space $(\R^n,d)$ is the $m$-dimensional Lebesgue measure $\mathcal{L}^m$ of the parallelotope generated by those vectors, i.e.,
\begin{equation*}
    \vol_m(a_1,\dots,a_m)\coloneq\mathcal{L}^m\left(\set{\lambda_1a_1+ \cdots + \lambda_m a_m : \lambda_i \in [0,1]}\right)
\end{equation*}
The \emph{minimal height} of $(a_1,\dots,a_m)$ is the smallest height of the parallelotope generated by those vectors, i.e.,
    \begin{equation*}
        \mh(a_1,\dots,a_m) \coloneq \min_{j \in \set{1,\dots,m}} d \left( a_j, \spn(a_1,\dots,\widehat{a}_j,\dots,a_m)\right) \, .
    \end{equation*}
Equivalently
\begin{align*}
        \mh(a_1,\dots,a_m) &= \min_{j \in \set{1,\dots,m}} \frac{\vol_m(a_1,\dots,a_m)}{\vol_{m-1}(a_1,\dots,\widehat{a}_j,\dots,a_m)} \\
        &= \frac{\vol_m(a_1,\dots,a_m)}{\max_{j \in \set{1,\dots,m}} \vol_{m-1}(a_1,\dots,\widehat{a}_j,\dots,a_m) }\,.
    \end{align*}
\end{definition} 

For the next result, one should think that $\Gamma$ is the projection, via the map $\pi$, of a geodesic. 

\begin{proposition} \label{mindelta}
    Let $\Gamma$ be a compact subset of the Euclidean space $(\R^n,d)$, $\delta>0$, and $m \in \N$ such that $\Gamma$ is not contained in the $\delta$-neighbourhood of any $\mathcal{P} \in \gr(n,m)$, i.e., for every $\mathcal{P} \in \gr(n,m)$ there exists $x \in \Gamma$ such that $d(x,P) \ge \delta$.
    Then there are $x_1,\dots,x_{m+1} \in \Gamma$ such that $\mh(x_1,\dots,x_{m+1})\ge \delta$.
\end{proposition}

\begin{proof}
    We firstly claim that there exists an $m$-tuple $(y_1,\dots,y_m)$ of points in $\Gamma$ such that $\vol_m(y_1,\dots,y_m)>0$. Indeed, by assumption $\spn(\Gamma)$ must have dimension grater than $m$. Hence, any linearly independent points $y_1,\dots,y_m$ in $\Gamma$ span positive $m$-volume. 
    
    Consider $(x_1,\dots,x_m)\in \Gamma^m$ that maximizes $\vol_m$ among all $m$-tuples in $\Gamma^m$, which exists by compactness and in particular $\vol_m(x_1,\dots,x_m)>0$. By hypothesis there exists $x_{m+1} \in \Gamma$ such that $d\left(x_{m+1},\spn(x_1,\dots,x_m)\right) \ge \delta$, i.e.,
    \begin{equation*}
      \frac{\vol_{m+1}(x_1,\dots,x_{m+1})}{\vol_m(x_1,\dots,x_m)}\ge \delta \, .  
    \end{equation*}
    For every $j\in \set{1,\dots,m+1}$, since $\vol_m(x_1,\dots,x_m)$ is maximal, we have
    \begin{align*}
        d(x_j,\spn(x_1,\dots,\widehat{x}_j,\dots,x_{m+1}) = \frac{\vol_{m+1}(x_1,\dots,x_{m+1})}{\vol_m(x_1,\dots,\widehat{x}_j,\dots,x_{m+1})}  \ge \frac{\vol_{m+1}(x_1,\dots,x_{m+1})}{\vol_m(x_1,\dots,x_m)}\ge \delta \, .
    \end{align*}
    This implies that $\mh(x_1,\dots,x_{m+1})\ge \delta$.
\end{proof}

\begin{proposition} \label{minspan}
    Let $a_1,\dots,a_m$ be vectors in $\R^n$, equipped with the Euclidean norm $\lVert \cdot \rVert$. Suppose that $\mh(a_1,\dots,a_m)>0$, then every $w \in \spn(a_1,\dots,a_m)$ can be written as $w = \lambda_1a_1 + \cdots + \lambda_m a_m$ with
    \begin{equation*}
         \lvert \lambda_j \rvert \le \frac{\lVert w \rVert}{\mh(a_1,\dots,a_m)} \, , \qquad \text{for every $j\in\set{1,\dots,m}$} \, .
    \end{equation*}
\end{proposition}
\begin{proof}
    Since $\mh(a_1,\dots,a_m)>0$, then $a_1,\dots,a_m$ are linearly independent and therefore every $w \in \spn(a_1,\dots,a_m)$ can be uniquely written as $w = \lambda_1a_1 + \cdots + \lambda_m a_m$. Let us now fix $j \in \set{1,\dots,m}$, we then have 
    \begin{align*}
        \lVert w \rVert &\ge d(w,\spn(a_1,\dots,\widehat{a}_j,\dots,a_m)     =d(\lambda_ja_j,\spn(a_1,\dots,\widehat{a}_j,\dots,a_m) \\
        &= \lvert \lambda_j \rvert \cdot d(a_j,\spn(a_1,\dots,\widehat{a}_j,\dots,a_m) \ge \lvert \lambda_j \rvert \cdot \mh(a_1,\dots,a_m) \, .
    \end{align*}
It follows that $\lvert \lambda_j \rvert \le \lVert w \rVert \cdot \mh(a_1,\dots,a_m)^{-1}$.
\end{proof}



The following is the key result:

\begin{proposition} \label{lippow}
    Let $\left(\G,\lVert \cdot \rVert_{\mathrm{sf}}\right)$ be a sub-Finsler free-Carnot group of step $2$ and rank $n$, denote by $\dcc$ its Carnot-Carathéodory distance. Fix on $\G$ an adequate scalar product and let $\deu$ denote the induced distance. Consider the sphere \eqref{def_seu} and its subsets \eqref{def_A_delta}. Then there exist a constant $C_5 \ge 1$ and a number $\delta_0 >0$ such that, for every $\delta \in (0,\delta_0)$, $g \in \sfeu \setminus A_\delta$, and $h \in \sfeu$, it holds
    \begin{equation*}
         \dcc(0_\G,h) \le \dcc(0_\G,g) +    \frac{C_5}{\delta} \cdot \deu(g,h) \, . 
    \end{equation*}
    In particular, the function $\dcc(0_\G,\cdot)$ is $\left(C_5\delta^{-1}\right)$-Lipschitz with respect to $\deu$ when restricted to $\sfeu \setminus A_\delta$.
\end{proposition}

Before the proof of \cref{lippow}, which is in \cref{proofs}, we need two preparatory lemmas:

\begin{lemma} \label{lemmawedge}
Consider on $\G$, defined as in \eqref{def_G}, an adequate scalar product and let $\lVert \cdot \rVert_{\mathrm{eu}}$ be the induced norm. Then there exists a constant $K_1$ such that, for every $(n-1)$-tuple $(x_1,\dots,x_{n-1})$ of points in $\R^n$, $Y \in \twowedge(\R^n)$, and $\varepsilon>0$ satisfying $$\mh(x_1,\dots,x_{n-1})\ge \varepsilon\, ,$$ there exist $v_1,\dots,v_{n-1} \in \R^n$ such that
 \begin{equation*}
        Y = \sum_{j=1}^{n-1} x_j \wedge v_j \qquad \text{and} \qquad \rVert v_j \lVert_{\mathrm{eu}} \le K_1\cdot\frac{\rVert Y \lVert_{\mathrm{eu}}}{\varepsilon} \, , \quad \text{for every $j \in \set{1,\dots,n-1}$.}
    \end{equation*}
\end{lemma}


\begin{proof}
    Fix $\varepsilon>0$ and an $(n-1)$-tuple $(x_1,\dots,x_{n-1})$ as in the assumptions and define
    \begin{equation*}
        E \coloneq \Set{\sum_{j=1}^{n-1} x_j \wedge v_j \, \colon \, \text{$v_j \in \R^n$ and $\lVert v_j \rVert_{\mathrm{eu}} \le \frac{1}{\varepsilon}$, for every $j \in \set{1,\dots,n-1}$}} \subseteq \twowedge\left(\R^n\right) \, .
    \end{equation*}
    We claim that it is sufficient to prove that $E$ contains an Euclidean ball in $\twowedge\left(\R^n\right)$ centered at zero with radius independent of $(x_1,\dots,x_{n-1})$ and $\varepsilon$. Indeed, if $B_\mathrm{eu}(0,r) \subseteq E$ and $0 \neq Y \in \twowedge\left(\R^n\right)$, then $\lambda Y \in E$ for $\lambda \coloneq r  \rVert Y \lVert_{\mathrm{eu}}^{-1}>0$, hence
    \begin{equation*}
        \lambda Y = \sum_{j=1}^{n-1} x_j \wedge v_j \qquad \text{for some $v_1,\dots,v_{n-1}$ such that $\lVert v_j \rVert_{\mathrm{eu}} \le \varepsilon^{-1}$.} 
    \end{equation*}
    It follows that
    \begin{equation*}
        Y = \sum_{j=1}^{n-1} x_j \wedge \frac{v_j}{\lambda} \qquad \text{for some $v_1,\dots,v_{n-1}$ such that $\left\lVert \frac{v_j}{\lambda} \right\rVert_{\mathrm{eu}} \le \frac{\rVert Y \lVert_{\mathrm{eu}}}{r \cdot \varepsilon}\, ,$}
    \end{equation*}
    and the proposition is proved by setting $K_1 \coloneq r^{-1}$. 
    

Fix an orthonormal basis $(e_1,\dots,e_n)$ of $\R^n$ such that $(e_1,\dots,e_{n-1})$ is an orthonormal basis of $\spn(x_1,\dots,x_{n-1})$. We claim that $\set{ e_i \wedge e_j \, | \, 1 \le i < j \le n }$, which is an orthonormal basis of $\twowedge \left(\R^n \right)$, is contained in $E$. Indeed, fix $Y\coloneq e_i \wedge e_j$ with $i<j$, then $i \le n-1$ and therefore $e_i \in \spn(x_1,\dots,x_{n-1})$. We recall that $\mh(x_1,\dots,x_{n-1})\ge \varepsilon$, therefore, by \cref{minspan}, we have that $e_i = a_i^1 x_1 + \dots + a_i^{n-1} x_{n-1}$ and
\begin{equation*}
    \left\lvert a_i^k \right\rvert \le \varepsilon^{-1}\left\lVert e_i \right\rVert_{\mathrm{eu}} = \varepsilon^{-1} \qquad \text{for every $k=1,\dots,n-1$} \, .
\end{equation*}
We can then write 
\begin{equation*}
    Y = e_i \wedge e_j = \sum_{k=1}^{n-1} \left(a_i^k x_k \right) \wedge e_j = \sum_{k=1}^{n-1} x_k \wedge \left(a_i^k e_j\right) \, ,
\end{equation*}
so that $Y = \sum_{k=1}^{n-1} x_k \wedge v_k$ with $v_k \coloneq a_i^k e_j$ and
\begin{equation*}
    \lVert v_k \rVert_{\mathrm{eu}} \le \left\lvert a_i^k \right\rvert \cdot \left\lVert e_j \right\rVert_{\mathrm{eu}} \le \varepsilon^{-1} \, .
\end{equation*}
We stress that $E$ is the image, through a linear map, of a symmetric and convex set. We infer that $E$  symmetric, i.e., $E=-E$, and convex as well. It is a general fact in Euclidean geometry that the convex hull of a symmetrised orthonormal basis contains an Euclidean ball of radius $1/\sqrt{d}$ where $d$ is the dimension of the space. 
\end{proof}


\begin{lemma} \label{lemmasize2}
    Let $\left(\G,\lVert \cdot \rVert_{\mathrm{sf}}\right)$ be a sub-Finsler free-Carnot group of step $2$ and rank $n$ as in \eqref{def_G}, denote by $\dcc$ its Carnot-Carathéodory distance. Fix on $\G$ an adequate scalar product and let $\lVert \cdot \rVert_{\mathrm{eu}}$ denote the induced norm. Then there exists a constant $K_2 > 0$ with the following property: 
    
    For every $\varepsilon>0$, every $Z \in \twowedge (\R^n)$, and every $n$-tuple $(g_0,\dots,g_{n-1}) \in \G^n$ such that $g_0=0_\G$ and  $\mh(\pi(g_1), \dots, \pi(g_{n-1}))\ge \varepsilon$, the following inequality holds:
    \begin{equation*}
        \dcc(g_0,Z * g_{n-1}) \le K_2 \cdot \frac{\lVert Z \rVert_{\mathrm{eu}}}{\varepsilon} + \sum_{i=1}^{n-1} \dcc(g_{j-1},g_j) \, .
    \end{equation*}
\end{lemma}

\begin{proof}
   Let $C_1$ and $K_1$ be the constants given in \eqref{constantc1} and \cref{lemmawedge}, respectively. We prove the result for $K_2\coloneq 2C_1K_1(n-1)^2$. Fix $\varepsilon$, $Z$, and an $n$-tuple $(g_0,\dots,g_{n-1})$ as in the assumptions. For every $j \in \set{0,\dots,n-1}$, we write $g_j=(x_j,Y_j)$ with $x_j \in \R^n$ and $Y_j \in \twowedge (\R^n)$. Since $g_0 = 0_\G$, we stress that $x_0=0 \in \R^n$. By hypothesis we have that $\mh(x_1,\dots,x_{n-1}) \ge \varepsilon$. By \cref{lemmawedge}, there exist $v_1,\dots,v_{n-1} \in \R^n$ and $K_1>0$, the latter being independent of $\varepsilon$, $Z$, and $(g_0,\dots,g_{n-1})$, such that
 \begin{equation} \label{zetavj}
        Z = \sum_{j=1}^{n-1} x_j \wedge v_j \qquad \text{and $\, \displaystyle{\rVert v_j \lVert_{\mathrm{eu}} \le K_1\cdot\frac{\rVert Z \lVert_{\mathrm{eu}}}{\varepsilon}}$, for every $j \in \set{1,\dots,n-1}$.}
    \end{equation}
For every $j \in \set{1,\dots,n-1}$, we define $w_j \coloneq - \sum_{k=j}^{n-1} v_k \in \R^n$ and we compute
\begin{align}
    Z &= \sum_{j=1}^{n-1} x_j \wedge v_j = \left[ \sum_{j=1}^{n-2} -v_j \wedge x_j  \right] - \left(v_{n-1} \wedge x_{n-1}\right) \nonumber \\ &= \left[ \sum_{j=1}^{n-2} (w_j-w_{j+1}) \wedge x_j  \right] + w_{n-1} \wedge x_{n-1} - w_1 \wedge x_0 \nonumber \\
    &= \sum_{j=1}^{n-1} w_j \wedge x_j - \sum_{j=0}^{n-2} w_{j+1} \wedge x_j = \sum_{j=1}^{n-1} w_j \wedge x_j - \sum_{j=1}^{n-1} w_j \wedge x_{j-1}  \nonumber \\
    &= \sum_{j=1}^{n-1} w_j \wedge (x_j - x_{j-1}) \, . \label{zetawedge}
\end{align}
Moreover, by \eqref{zetavj}, we have
    \begin{equation} \label{normwj}
        \lVert w_j \rVert_{\mathrm{eu}} \le (n-1)K_1\cdot \frac{\lVert Z \rVert_{\mathrm{eu}}}{\varepsilon}\, , \qquad \text{for every $j \in \set{1,\dots,n-1}$.}
    \end{equation}

Fix $j \in \set{1,\dots,n-1}$, we observe that $$g_{j-1}^{-1}*g_j\stackrel{\eqref{prod}}{=}\left( x_j - x_{j-1}, Y_j - Y_{j-1} + \frac{1}{2}(x_j \wedge x_{j-1})\right) \, . $$
    Using again the product law \eqref{prod}, we compute
    \begin{align}
        w_j&*g_{j-1}^{-1}*g_j*(-w_j) = w_j*\left( x_j - x_{j-1} + Y_j - Y_{j-1} + \tfrac{1}{2}(x_j \wedge x_{j-1})\right)*(-w_j) \nonumber \\
        &= \left( x_j - x_{j-1} + w_j + Y_j - Y_{j-1} + \tfrac{1}{2}(x_j \wedge x_{j-1})+ \tfrac{1}{2}(w_j \wedge (x_j - x_{j-1})\right)*(-w_j) \nonumber \\
        &=  x_j - x_{j-1}+ Y_j - Y_{j-1} + \tfrac{1}{2}(x_j \wedge x_{j-1})+ \tfrac{1}{2}(w_j \wedge (x_j - x_{j-1}))-\tfrac{1}{2}( (x_j - x_{j-1})\wedge w_j) \nonumber \\
        &= x_j - x_{j-1} + Y_j - Y_{j-1} + \tfrac{1}{2}(x_j \wedge x_{j-1})+ w_j \wedge (x_j - x_{j-1}) \nonumber\\
        &= g_{j-1}^{-1}*g_j*\left(w_j \wedge (x_j - x_{j-1})\right) \, . \label{chaineq}
    \end{align}
We recall that $\twowedge (\R^n)$ is contained in the center of $\G$ and that  $g_0=0$, so that we can write
\begin{align*}
    Z * g_{n-1} & \stackrel{\eqref{zetawedge}}{=} \left(\sum_{j=1}^{n-1} w_j \wedge (x_j - x_{j-1})\right)* \left(g_{0}^{-1}*g_1\right)*\dots* \left(g_{n-2}^{-1}*g_{n-1}\right)  \\
    &=  g_{0}^{-1}*g_1*\left(w_1 \wedge (x_1 - x_{0})\right)* \dots *  g_{n-2}^{-1}*g_{n-1}*\left(w_{n-1} \wedge (x_{n-1} - x_{n-2})\right)  \\
    &\stackrel{\eqref{chaineq}}{=}  w_1*g_{0}^{-1}*g_1*(-w_1)*\dots*w_{n-1}*g_{n-2}^{-1}*g_{n-1}*(-w_{n-1}) \, .
\end{align*}
Therefore, since the Carnot-Carathéodory distance $\dcc$ is left invariant:
\begin{align*}
    \dcc\left(0_\G, Z*g_{n-1}\right) &\le \sum_{j=1}^{n-1} \left(\dcc(0_\G,w_j)) + \dcc\big(0_\G, g_{j-1}^{-1}*g_j\big) + \dcc(0_\G,-w_j)\right) \\
    & \stackrel{\eqref{c1_estiamate}}{\le} 2 \sum_{j=1}^{n-1} C_1 \lVert w_j \rVert_\mathrm{eu} + \sum_{j=1}^{n-1} \dcc(g_{j-1},g_j) \\
    & \stackrel{\eqref{normwj}}{=} 2C_1(n-1)^2K_1 \cdot \frac{\lVert Z \rVert_{\mathrm{eu}}}{\varepsilon} + \sum_{j=1}^{n-1} \dcc(g_{j-1},g_j) \, .
\end{align*}
The statement is then proved by setting $K_2\coloneq 2C_1K_1(n-1)^2$.
\end{proof}


\subsection{Final proofs} \label{proofs}
We are ready for the proof of \cref{lippow} and finally of \cref{rectifiability}.

\begin{proof}[Proof of \cref{lippow}]
Let $C_1$, $C_4$, and $K_2$ be the constants given in \eqref{constantc1}, \cref{abndist}, and \cref{lemmasize2}, respectively. We prove the result for $\delta_0 \coloneq 2K_2/(C_1C_4)$ and $C_5 \coloneq \max\set{1,4K_2/C_4}$. Fix $g \in \sfeu \setminus A_\delta$ and consider a geodesic $\gamma \colon [0,1] \to \G$ such that $\gamma(0)=0_\G$ and $\gamma(1) = g$. Since $g \in \sfeu \setminus A_\delta$ , then $\deu(\gamma(1),\abn(\G))\ge\delta$ and, by \cref{abndist}, there exists $C_4>0$ such that the set $\Gamma \coloneq \Set{\pi(\gamma(t)) \, | \, t \in [0,1]}$ is not contained in the $C_4\delta$-neighbourhood of any $P \in \gr(\R^n,n-2)$. Consequently, by \cref{mindelta}, there are $x_i \in \Gamma$, for $i\in \set{1,\dots,n-1}$, such that \begin{equation} \label{mhcondition}
  \mh(x_1,\dots,x_{n-1})\ge C_4\delta \, .  
\end{equation} Since the minimal height does not depend on the order, we can further assume that $x_j = \pi(\gamma(t_j))$ for $0 \le t_1 < \dots < t_{n-1}$.

We set $g_i\coloneq \gamma(t_i)$, for $i \in \set {1,\dots,n-1}$, and $g_0 \coloneq \gamma(0)=0_\G$. We also now fix $h \in \sfeu$. We write $g$ and $h$ in coordinates:
\begin{equation*}
    g = (\xi_1,W_1) \, , \quad h = (\xi_2,W_2) \, .
\end{equation*}
We then define 
\begin{equation*}
    q \coloneq \xi_2-\xi_1 \in \R^n  \, , \quad Z \coloneq W_2-W_1+\tfrac{1}{2}(\xi_2 \wedge \xi_1) \in \twowedge(\R^n) \, .
\end{equation*}
We have 
\begin{equation} \label{dccq}
   \dcc(0_\G,q) \stackrel{\eqref{c1_estiamate}}{\le} C_1 \lVert \xi_2 - \xi_1 \rVert_\mathrm{eu}  \le C_1 \deu(g,h) \, .
\end{equation}
Moreover, we also have
\begin{align}
   \lVert Z \rVert_\mathrm{eu} & \le \lVert W_2 - W_1 \rVert_\mathrm{eu} + \frac{1}{2} \lVert (\xi_2-\xi_1)\wedge \xi_1 \rVert_\mathrm{eu} \nonumber \\ &\le  \deu(g,h) + \frac{1}{2} \lVert \xi_2 - \xi_1 \rVert_\mathrm{eu} \cdot \lVert \xi_1 \rVert_\mathrm{eu} \nonumber \\ &\le  \deu(g,h) + \frac{\lVert \xi_1 \rVert_\mathrm{eu}}{2} \deu(g,h) \nonumber \\ &\le 2\cdot \deu(g,h)  \, , \label{dccz}
\end{align}
where we used that the scalar product is adequate and that $g \in \sfeu$.

By the product rule \eqref{prod}, we observe that $h=Z *g *q$. We recall that $\dcc$ is left-invariant, so that we infer
\begin{equation} \label{dcch}
    \dcc(0_\G,h)\le\dcc(0_\G,Z*g)+\dcc(Z*g,Z*g*q)\le\dcc(0_\G,Z*g_{n-1})+\dcc(g_{n-1},g)+\dcc(0_\G,q) \, .
\end{equation}
Thanks to \eqref{mhcondition}, the $n$-tuple $(g_0,\dots,g_{n-1})$ satisfies the conditions of \cref{lemmasize2} with $\varepsilon$ replaced by $C_4\delta$, so we estimate the first term of \eqref{dcch}:
\begin{equation*}
    \dcc(0_\G,Z*g_{n-1}) \le K_2 \cdot \frac{\lVert Z \rVert_\mathrm{eu}}{C_4\delta} + \sum_{i=1}^{n-1} \dcc(g_{j-1},g_j) \, .
\end{equation*} 
We shall use the previous inequality in \eqref{dcch} and use that $g_0,\dots,g_{r-1},g$ are consecutive points along the geodesic $\gamma$. If $\delta \le \delta_0 \coloneq 2K_2/(C_1C_4)$, we get
\begin{align*}
   \dcc(0_\G,h) &\le  K_2 \cdot \frac{\lVert Z \rVert_\mathrm{eu}}{C_4\delta} + \sum_{i=1}^{n-1} \dcc(g_{j-1},g_j) + \dcc(g_{n-1},g)+\dcc(0_\G,q) \\ &= \dcc(0_\G,g) + \frac{K_2}{C_4\delta}\cdot \lVert Z \rVert_\mathrm{eu} + \dcc(0_\G,q) \\ 
   &\stackrel{\eqref{dccq}\eqref{dccz}}{\le} \dcc(0_\G,g) + \frac{2K_2}{C_4\delta}\deu(g,h) + C_1\deu(g,h) \\ &\le \dcc(0_\G,g) + \frac{4K_2}{C_4\delta}\deu(g,h) \, .
\end{align*}
The statement is then proved by setting $C_5 \coloneq \max\set{1,4K_2/C_4}$.
\end{proof}

Here is the proof of the main result:


\begin{proof}[Proof of \cref{rectifiability}]
We see $\G$ as in \eqref{def_G} with an Euclidean distance as in \cref{setting}, denoted by $\deu$, so to have the sphere $\sfeu$ as in \eqref{def_seu}. 
To prove the Euclidean rectifiability of $\sfcc$, since $\sfeu$ is locally bi-Lipschitz equivalent to $\R^{d-1}$,
we consider the sets $E_n$ for \eqref{eq_def_rect} to be subsets of $\sfeu $.
With the notation \eqref{def_A_delta}, we then define 
\begin{align*}
    E_n \coloneq A_{2^{-n}} \setminus A_{2^{-n-1}} &= \set{x \in \sfeu \, | \, 2^{-n-1} < \deu(x,\abn(\G)\le 2^{-n}} \, , \quad \forall n \ge 1 \quad \text{and}
    \\ E_0 \coloneq \sfeu \setminus A_{1/2} &= \set{x \in \sfeu \, | \, \deu(x,\abn(\G) > \tfrac{1}{2}} \, .
\end{align*}
We denote by $  \mathcal{H}^{d-1}=\mathcal{H}_{\mathrm{Riem}}^{d-1}$ the $(d-1)$-dimensional Hausdorff measure with respect to the distance $\deu$.
We obviously have that $\mathcal{H}^{d-1}(E_0) \le \mathcal{H}^{d-1}(\sfeu)$.
By \cref{dimabn} we also get \begin{equation}\label{stima_C4}
    \mathcal{H}^{d-1}(E_n) \le \mathcal{H}^{d-1}(A_{2^{-n}}) \le C_3 \cdot 2^{-3n}. 
\end{equation}

Let $d_0(\cdot)\coloneq  \dcc(0_\G,\cdot)$ and $\phi$    be  the dilating map from \cref{diffeophi}.
We consider a sequence of functions $(f_n)_{n\in \N}$, where $f_n$ is the image by    $\phi$ of the graphing map of $d_0$ restricted to $E_n$. Namely, 
\begin{equation*}
    f_n \colon E_n \to \G \qquad  x \mapsto f_n(x)\coloneq \phi(x,d_0(x)) 
    =\delta_{d_0(x)}^{-1}(x).
\end{equation*}
Our aim is to show that the maps $\{f_n\}_n$ give the  Euclidean rectifiability of $\sfcc$. 

 We first explain why the maps $f_n$ are Lipschitz. Notice that each of these maps is the composition between the graph map
 $x \mapsto  (x,d_0(x)) $, which has Lipschitz constant $\left(1+ \mathrm{Lip}(d_0|_{E_{n}})\right)$, and the diffeomorphism $\phi$. Moreover, the graph of $d_0$ is contained in a compact set, on which therefore $\phi$ is Lipschitz with constant $L$. We use \cref{lippow}, recalling that $C_5 \ge 1$, to get 
  \begin{equation} \label{lipn}
     \mathrm{Lip}(f_{n\,|E_n}) \le L\left(1+ \mathrm{Lip}(d_0|_{E_{n}}) \right) \le L\left(1+C_5 \cdot 2^{n+1}\right)\leq  C_5L \cdot 2^{n+2}. 
  \end{equation}
  Thus, the maps $f_n$ are Lipschitz. 

  We next prove that the sets $f_n(E_n)$ cover almost all of $\sfcc$. Since $\bigcup^\infty_{n=0} E_n = \sfeu \, \setminus \, \abn(\G)$ and each function $f_n$ is a dilation map, by recalling  that $\abn(\G)$ is invariant under dilations, we get 
  \begin{equation}\label{eq_sfera_dentro}
      \sfcc \,  \setminus \, \bigcup^\infty_{n=0} f_n(E_n) = \sfcc \cap \abn(\G) \subseteq \abn(\G) \, .
  \end{equation}
  Moreover, from \cref{codim3} we have that the Hausdorff dimension of $\abn(\G)$ is at most $d-3$ and therefore $\mathcal{H}^{d-1}(\abn(\G))=0$. Thus
\begin{equation}\label{complementare_nullo}
            \mathcal{H}^{d-1}\left( \sfcc \,  \setminus \, \bigcup^\infty_{n=0} f_n(E_n) \right)
            \stackrel{\eqref{eq_sfera_dentro}}{\le} \mathcal{H}^{d-1}\left(\abn(\G) \right) =0 \, .
        \end{equation}
        
It is left to prove that   $ \mathcal{H}^{d-1}\left( \sfcc   \right)$ is finite.
Because of \eqref{complementare_nullo}, it is then sufficient to prove 
\begin{equation}\label{misura_finita}
    \mathcal{H}^{d-1}\left(\bigcup^\infty_{n=0} f_n(E_n)\right)=\sum_{n=0}^\infty \mathcal{H}^{d-1}\left(f_n(E_n)\right)  <\infty \, .
\end{equation}
We then check that we have a summable sequence: 
\begin{align*}
\mathcal{H}^{d-1}\left(f_n(E_n)\right) 
&= \mathcal{H}^{d-1}\left(\phi(\mathrm{gr}   (d_0|_{E_{n}}))\right) \\
& \le C_2 \cdot \mathcal{H}^{d-1}_{\mathrm{eu}}(E_n) \cdot \left( \mathrm{Lip}(d_{0|_{E_{n}}}) + 1\right) \\ &\stackrel{\eqref{stima_C4}}{\le} C_2 C_3\cdot 2^{-3n} \cdot \left( \mathrm{Lip}( d_{0|_{E_{n}}}) + 1\right) \\ &\stackrel{\eqref{lipn}}{\le} C_2C_3C_5L \cdot 2^{-3n} \cdot 2^{n+2} = C_2C_3C_5L \cdot  2^{-2n+2} \, , 
\end{align*}
where in the first inequality  we used \cref{areaformula}. Hence, \eqref{misura_finita} is obtained and the proof of the rectifiability of $\sfcc$
is concluded.
\end{proof}

\begin{remark}
    We observe that we could have obtained the convergence of the series in \eqref{misura_finita} even with weaker versions of \cref{dimabn} and \cref{lippow} which we used in \eqref{stima_C4} and \eqref{lipn}, respectively. As an example, with could either work in a setting where, for $\varepsilon>0$, the abnormal set has finite Minkowski content of co-dimension $1+\varepsilon$, or a setting where the Lipschitz constant of $\dcc$ behaves like the $\varepsilon-3$ power of the distance from the abnormal set, or a combination of the two.
\end{remark}

\subsection{A sharp example} We conclude providing an example that shows the optimality of \cref{lippow} in the following sense:

\begin{proposition} \label{optimality}
  There exist a sub-Finsler free-Carnot group $\G$ of step $2$, equipped with an adequate scalar product, and numbers $C,\delta_0 >0$ such that for every $0 < \delta \le \delta_0$ there exists $g_\delta \in \sfeu \setminus A_\delta$ with the property that, for every $\varepsilon<0$, there exists $h_{\varepsilon,\delta} \in \sfeu$ with $0< \deu(g_\delta,h_{\varepsilon,\delta})\le \varepsilon$ such that
\begin{equation*}
         \dcc(0_\G,h_{\varepsilon,\delta}) \ge \dcc(0_\G,g_\delta) +    \frac{C}{\delta} \cdot \deu(g_\delta,h_{\varepsilon,\delta}) \, . 
    \end{equation*}  
\end{proposition} 

We firstly claim that, to prove \cref{optimality}, it not restrictive to consider $g_\delta$ and $h_{\varepsilon,\delta}$ to be in the $\frac{1}{2}$-tubular neighbourhood (with respect to the metric $\deu$) of $\sfeu$ instead of being in $\sfeu$. Indeed, their projections into $\sfeu$, through the map $\sigma$ defined as in \eqref{map_pi}, will satisfy the same inequality with a modified constant $C$, not depending on $g_\delta$ and $h_{\varepsilon,\delta}$. The proof of \cref{optimality} is then given by the following example.

\begin{example} 
Consider the free-Carnot group $\G \coloneq \R^4 \times \twowedge (\R^4)$ of step $2$ and rank $4$, and $\set{e_1,e_2,e_3,e_4}$ to be the standard basis of $\R^4$. Consider on $\G$ the adequate scalar product such that $e_1,e_2,e_3,e_4$ are orthonormal. Consider $\lVert \cdot \rVert_{\mathrm{sf}}$ to be the $\ell^1$ norm with respect to the basis $\set{e_1,e_2,e_3,e_4}$. For every $\varepsilon,\delta \ge 0$, we define
\begin{equation*}
    g_\delta \coloneq e_1 \wedge e_2 + \delta e_3 \, , \qquad h_{\varepsilon,\delta} \coloneq g_\delta + \varepsilon e_3 \wedge e_4 \, .
\end{equation*}
It is easy to show that, if $\delta \le 1/2$, then $\deu(g_\delta,\abn(\G))=\delta$ and $1\le \deu(0_\G,g_\delta) \le 3/2$. 
Let us consider $\gamma \colon [0,1] \to \G$ to be a geodesic with $\gamma(0)=0_G$ and $\gamma(1)=h_{\varepsilon,\delta}$, denote by $u \in L^1([0,1],\g_1)$ the control associated to $\gamma$. Define the map $\pi_{12} \colon \g_1 \to \spn\set{e_1,e_2}$ to be the projection parallel to $\spn\set{e_3,e_4}$ and $\pi_{34} \coloneq \mathrm{id}_{\g_1} - \pi_{12}$. Let us define $u_{12}\coloneq\pi_{12}(u)$ and $u_{34}\coloneq \pi_{34}(u)$. It easily follows that
\begin{equation*}
    \gamma_{u_{12}}(1)= e_1 \wedge e_2 \, , \qquad \gamma_{u_{34}}(1) = \delta e_3  + \varepsilon e_3 \wedge e_4 \, .
\end{equation*}

Let us define the control $\widetilde{u} \colon [0,2] \to \g_1$ as
\begin{equation*}
    \widetilde{u}(t) \coloneq \begin{cases}
        u_{12}(t) &\quad \text{if $t \in [0,1]$,} \\ u_{34}(t-1) &\quad \text{if $t \in [1,2]$.}
    \end{cases}
\end{equation*}
We then observe that
\begin{align*}
    \ell(\gamma_{\widetilde{u}})= \int^1_0 \lvert u_{12}(t) \rvert \, \de t + \int^1_0 \lvert u_{34}(t) \rvert \, \de t = \int^1_0 \lvert u_{12}(t) +u_{34}(t) \rvert \, \de t = \int^1_0 \lvert u(t) \rvert \, \de t = \ell(\gamma_u) \, ,
\end{align*}
so that $\gamma_{\widetilde{u}}$ is again a geodesic from $0_\G$ to $h_{\varepsilon,\delta}$. Moreover $\gamma_{\widetilde{u}}(1)=e_1 \wedge e_2$ and $\gamma_{\widetilde{u}}(2)=(e_1 \wedge e_2)*(\delta e_3 + \varepsilon e_3 \wedge e_4)$. Since $e_1 \wedge e_2$ is in the center of $\G$ we also have that $\dcc(\gamma_{\widetilde{u}}(1), \gamma_{\widetilde{u}}(2))= \dcc(0_\G, \delta e_3 + \varepsilon e_3 \wedge e_4)$. The latter proves that $\dcc(0_\G,h_{\varepsilon,\delta}) = \dcc(0_G, e_1 \wedge e_2) + \dcc(0_\G, \delta e_3 + \varepsilon e_3 \wedge e_4)$, so that
\begin{equation} \label{example}
    \dcc(0_\G,h_{\varepsilon,\delta}) - \dcc(0_\G,g_\delta) = \dcc(0_\G, \delta e_3 + \varepsilon e_3 \wedge e_4) - \dcc(0_\G, \delta e_3) \, ,
\end{equation}
and to compute the terms in the right hand side it is not restrictive to consider geodesics contained in $\spn\set{e_3,e_4,e_3 \wedge e_4}$, that is isometric to the Heisenberg group equipped with the $\ell^1$ norm on the horizontal distribution. Finally, we remark that $\deu(h_{\varepsilon,\delta},g_\delta)=\varepsilon$ and that the complete classification of geodesics for the Heisenberg group equipped with the $\ell^1$ norm (see \cite{rate}) allows us to compute explicitly
\begin{equation*}
    \dcc(0_\G, \delta e_3) = \delta \, , \qquad \dcc(0_\G, \delta e_3 + \varepsilon e_3 \wedge e_4) = \delta + 2\frac{\varepsilon}{\delta} \qquad \text{for every $\varepsilon \le \delta^2$.} 
\end{equation*}
By combining the latter with \eqref{example}, we conclude that
\begin{equation*}
    \dcc(0_\G,h_{\varepsilon,\delta}) = \dcc(0_\G,g_\delta) + \frac{2}{\delta}\varepsilon = \dcc(0_\G,g_\delta) + \frac{2}{\delta} \deu(h_{\varepsilon,\delta},g_\delta) \, .
\end{equation*}

\end{example}

\bibliographystyle{abbrv}
\bibliography{biblio}

\end{document}